\documentclass[10pt,reqno]{amsart}

\usepackage{hyperref}
\usepackage{float}
\usepackage{mathtools}
\usepackage[all]{xy}

\usepackage{tikz-cd}
\usepackage{quiver}
\usepackage{amscd,amssymb,amsmath,latexsym,bm}
\usepackage[mathcal,mathscr]{euscript}
\usepackage{lipsum}
\usepackage{amsfonts}
\usepackage{graphicx}
\usepackage{epstopdf}
\usepackage{algorithmic}
\usepackage{xcolor}
\usepackage{upgreek}
\usepackage{enumerate}
\usepackage{ulem}
\usepackage{stmaryrd}
\usepackage{gensymb}			
\usepackage[utf8]{inputenc}
\usepackage[width=1\linewidth]{caption}

\newtheorem{theorem}{Theorem}[section]

\newtheorem{proposition}[theorem]{Proposition}
\newtheorem{remark}[theorem]{Remark}
\newtheorem{definition}[theorem]{Definition}
\newtheorem{example}[theorem]{Example}

\numberwithin{equation}{section}
\numberwithin{figure}{section}

\newcommand{\CM}{{\mathbb C}}
\newcommand{\NM}{{\mathbb N}}

\newcommand{\RM}{{\mathbb R}}

\newcommand{\TM}{{\mathbb T}}
\newcommand{\ZM}{{\mathbb Z}}

\newcommand*{\ev}{\mathrm{ev}}

\providecommand{\abs}[1]{\left \lvert#1 \right \rvert} 
\providecommand{\norm}[1]{\left \lVert#1 \right \rVert}
\DeclarePairedDelimiter{\tnorm}{\lvert\!\lvert\!\lvert}{\rvert\!\rvert\!\rvert}

\newcommand{\Aa}{{\mathcal A}}

\newcommand{\Bb}{{\mathcal B}}

\newcommand{\Gg}{{\mathcal G}}

\newcommand{\Uu}{{\mathcal U}}

\newcommand{\Xx}{{\mathcal X}}

\begin{document}
\title[]{Spectral continuity for étale groupoids with the Rapid decay property}

\author{Tom Stoiber}

\address{Department of Mathematics, University of California, Irvine, CA 92717, USA \\
	\href{mailto:tstoiber@uci.edu}{tstoiber@uci.edu}}

\date{\today}

\begin{abstract}
We show that the reduced groupoid $C^*$-algebras of continuous fields of étale groupoids satisfying the rapid decay property yield continuous fields of $C^*$-algebras. This establishes a new sufficient criterion that applies in the non-amenable case where the full and reduced groupoid algebras may differ. Potential applications include convergence of spectra in inverse systems of finite-index subgroups and magnetic models on hyperbolic lattices.
\end{abstract}

\maketitle

\setcounter{tocdepth}{1}

\section{Introduction and Main Statements}
\label{Sec:Introduction}

A field of groupoids (sometimes also referred to as a groupoid bundle) is a topological groupoid $\Gg$ which decomposes into a disjoint union $\Gg:=\sqcup_{t\in T} \Gg_t$ of subgroupoids parametrized by a topological space $T$, but which is generally not supplied with the disjoint union topology. They appear naturally in many contexts such as pseudodifferential calculus (e.g. via the tangent groupoid \cite{ConnesBook94,Higson2008}), twisted group algebras and topological dynamics (see \cite{Beckus18} and references therein). In those cases one often also deals with families of self-adjoint operators in or affiliated to the reduced groupoid algebras $C^*_r(\Gg_t)$ of the fibers and wants to study the dependence of their spectral or topological data (e.g. K-theory classes) on the parameter. 

A central question in this context is whether the reduced algebra $C^*_r(\Gg)$ decomposes as a continuous field of $C^*$-algebras with fibers $C_r^*(\Gg_t)$. This is a very useful property because it implies that for any continuous section of self-adjoint operators the spectra depend continuously on $t$ in the Hausdorff metric. A systematic study of this question was undertaken in \cite{LandsmanContMath2001} and it was found that when $\Gg$ is amenable, then its reduced groupoid $C^*$-algebra indeed defines a continuous field:

\begin{theorem}[{\cite[Theorem 5.5]{LandsmanContMath2001}}]
\label{th:amenable}
Let $\Gg$ be a locally compact, second-countable Hausdorff groupoid with left-continuous Haar system forming a field of groupoids over a locally compact Hausdorff space $T$ with fibers $(\Gg_t)_{t\in T}$. If $\Gg$ is amenable then full and reduced groupoid algebras coincide $C^*_r(\Gg) = C^*(\Gg)$ and decompose as a continuous field of $C^*$-algebras with fibers $C^*_r(\Gg_t)$.
\end{theorem}
The analogous statement also holds for groupoid algebras which are twisted by a continuous $2$-cocycle \cite{Beckus18}.  There are counterexamples where $C_r^*(\Gg)$ cannot be a continuous field with fibers $C_r^*(\Gg_t)$, in particular HLS groupoids \cite{HigsonLafforgueSkandalis2002} as we recall in Section~\ref{sec:examples}. In contrast, no readily checkable sufficient criteria seem to be known in the non-amenable case.

In this paper we provide one such condition:
\begin{theorem}
\label{th:main}
Let $\Gg$ be a locally compact, second-countable Hausdorff \'etale groupoid twisted by a 2-cocycle $\omega$ and which is a field of groupoids over a locally compact Hausdorff space $T$ with fibers $(\Gg_t)_{t\in T}$. If there exists a length function on $\Gg$ such that $\Gg$ has the $\omega$-twisted rapid decay property then $C_r^*(\Gg,\omega)$ is a continuous field of $C^*$-algebras with fibers $C_r^*(\Gg_t, \omega\rvert_{\Gg_t})$.
\end{theorem}
The generalization of the rapid decay property from the group case \cite{Jolissaint} to (twisted) groupoids has been developed in \cite{Hou,Weygandt} and, in particular, is also satisfied by a large class of groupoids that are not amenable.

The structure of the paper is as follows. In Section~\ref{Sec:Definitions} we recall the necessary definitions regarding fields of groupoids and $C^*$-algebras. The proof of Theorem~\ref{th:main} then consists of two steps: In Section~\ref{Sec:Continuity} we use functional calculus to reduce the proof of continuity to existence of a dense subset of sections which is closed under smooth functional calculus and which admits zero-preserving approximations. In Section~\ref{Sec:RapidDecay}, we then recall facts about the twisted rapid decay property for groupoids and show that it implies the sufficient conditions developed in the previous section. The final Section~\ref{sec:examples} gives examples illustrating the range of applicability of our result, in particular HLS groupoids and spectral continuity of the discrete magnetic Laplacian on hyperbolic crystals.
\vspace{0.1cm}

\noindent{\small{{\bf Acknowledgements}\;\; The author thanks Emil Prodan for helpful discussions. This work was partially supported by the grants NSF DMS-2052899, DMS-2155211, Simons-681675, as well as the German research foundation Project-ID 521291358.}}

\section{Definitions}
\label{Sec:Definitions}

Throughout this paper, $\Gg$ will denote a topological Hausdorff groupoid which is locally compact, second-countable (lcsc) and \'etale. The set of composable elements is denoted by $\Gg^{(2)}\subset \Gg\times \Gg$ and as usual $\Gg_u=s^{-1}(\{u\})$ and $\Gg^u=r^{-1}(\{u\})$ for elements of the unit space $u\in \Gg^{(0)}$. Recall that a topological groupoid is \'etale if the source and range maps $s,r: \Gg\to \Gg^{(0)}$ are local homeomorphisms. An open bisection $U\subset \Gg$ is an open subset such that the restrictions of the source and range map are both homeomorphisms onto open subsets of $\Gg^{(0)}$. Since $\Gg$ is \'etale the open bisections form a basis of its topology and the counting measures define a canonical Haar system \cite[Proposition I.2.8]{RenaultBook}. 

We allow twists by a continuous groupoid $2$-cocycle $\omega: \Gg^{(2)} \to \Uu(\CM)\simeq \TM$, i.e. a continuous function satisfying the cocycle condition
$$\omega(\gamma_1,\gamma_2)\omega(\gamma_1\gamma_2,\gamma_3)=\omega(\gamma_1,\gamma_2\gamma_3)\omega(\gamma_2,\gamma_3)$$
for all $\gamma_i\in \Gg$ such that $(\gamma_1,\gamma_2),(\gamma_2,\gamma_3)\in \Gg^{(2)}$. The twisted convolution algebra $C_c(\Gg,\omega)$ supplies the compactly supported continuous functions with the algebraic structures \cite{RenaultBook}
\begin{align}\label{eq:convolution}
(f_1*f_2)(\gamma)&=\sum_{\substack{\eta\in \Gg\\ s(\eta)=s(\gamma)}} \omega(\gamma\eta^{-1},\eta) f_1(\gamma\eta^{-1}) f_2(\eta)\\
(f^*)(\gamma)&=\overline{\omega(\gamma,\gamma^{-1})} \,\overline{f(\gamma^{-1})}.
\end{align}
For each element $u\in \Gg^{(0)}$ of the unit space we have a representation $\pi^\omega_u$ of the twisted convolution algebra $C_c(\Gg, \omega)$ on $\ell^2(\Gg_u)$ given by the formula
$$(\pi^\omega_u(f)\xi)(\gamma) =\sum_{\eta \in \Gg_u} \omega(\gamma\eta^{-1},\eta) f(\gamma\eta^{-1})\xi(\eta).$$
The reduced groupoid $C^*$-algebra $C_r^*(\Gg,\omega)$ is the completion of $C_c(\Gg, \omega)$ w.r.t. the norm
$$\lVert f\rVert = \sup_{u \in \Gg^{(0)}} \lVert \pi^\omega_u(f)\rVert. $$
On $C_c(\Gg)$ one also considers the pair of $L^1$-norms 
$$\norm{f}_{I,s}=\sup_{u \in \Gg^{(0)}} \sum_{\gamma \in \Gg_u} \abs{f(\gamma)}, \qquad \norm{f}_{I,r}=\sup_{u \in \Gg^{(0)}} \sum_{\gamma \in \Gg^u} \abs{f(\gamma)}.$$
For the reduced norm one has $$\norm{f}_{C^*_r(\Gg,\omega)}\leq \norm{f}_I:=\max(\norm{f}_{I,s},\norm{f}_{I,r}).$$ The full $C^*$-algebra $C^*(\Gg,\omega)$ is defined by completing $C_c(\Gg,\omega)$ by the supremum over all norms of representations that are dominated by $\norm{\cdot}_I$. 

\begin{definition}
\label{def:cont_field_groupoids}
A field of groupoids over a topological space $T$ is a tuple $(\Gg,T,p)$ consisting of a topological groupoid $\Gg$ with $p:\Gg \to T$ a continuous open surjection such that $p = p \circ s = p \circ r$.
\end{definition}
The fibers of a field $\Gg_t=p^{-1}(t)$ are closed subgroupoids and as a set $\Gg$ is just the disjoint union $\sqcup_{t\in T} \Gg_t$.

For continuous fields of $C^*$-algebras slightly different variant definitions exist, but they are mostly equivalent up to differences in terminology. To be compatible with \cite{LandsmanContMath2001} we follow \cite{KirchbergWassermann1995}:
\begin{definition}
\label{def:cont_field_algebras}
Let $T$ be a locally compact Hausdorff space.  
A bundle of $C^{*}$-algebras over $T$ is a triple $\bigl(\mathcal \Aa,\{\Aa_{t}\}_{t\in T},\{\ev_{t}\}_{t\in T}\bigr)$ consisting of $C^*$-algebras $\Aa$ and $\Aa_t$ as well as surjective $*$-homomorphisms $\ev_{t}:\mathcal \Aa \to \Aa_{t}$, such that
\begin{itemize}
  \item[(i)] there is an action of $C_0(T)$ on $\Aa$ such that $\ev_t(f\cdot a) = f(t)\ev_t(a)$ for all $f\in C_{0}(T)$, $a\in \Aa$ and $t\in T$;
  \item[(ii)] for all $a\in\mathcal \Aa$ one has $\norm{a}=\sup_{t\in T}\norm{\ev_{t}(a)}$.
\end{itemize}
A bundle of $C^*$-algebras is a continuous field if in addition the function
\begin{equation}
\label{eq:normfunction} t\in T \mapsto \lVert  \ev_t(a)\rVert
\end{equation}
is in $C_0(T)$ for each $a\in \Aa$.
\end{definition}

An important consequence of the continuity of \eqref{eq:normfunction} is spectral continuity:

\begin{proposition}[{\cite{Beckus18}}]
\label{prop:spec_cont}
For a continuous field of $C^*$-algebras $\bigl(\mathcal \Aa,\{\Aa_{t}\}_{t\in T},\{\ev_{t}\}_{t\in T}\bigr)$ over a locally compact Hausdorff space $T$ and any self-adjoint element $h\in \Aa$ the map
	$$t\in T \mapsto \sigma(\ev_t(h)) \subset{\RM}$$
	is continuous w.r.t the Hausdorff metric on the compact subsets of $\RM$.
\end{proposition}

Conversely, a triple $\bigl(\mathcal \Aa,\{\Aa_{t}\}_{t\in T},\{\ev_{t}\}_{t\in T}\bigr)$ which does not satisfy this spectral continuity cannot be a continuous field of $C^*$-algebras.

\section{Continuity from dense sets of sections}
\label{Sec:Continuity}
We will construct continuous fields from a suitable dense subset of sections:
\begin{proposition}
\label{prop:fields_by_sections}
Let $(\Aa_t)_{t\in T}$ be $C^*$-algebras parametrized by a locally compact Hausdorff space $T$. Consider a set $\Xx \subset \prod_{t\in T}\Aa_t$  with the following properties:
\begin{enumerate}
	\item[(i)]  $\Xx$ is a $*$-subalgebra of $\prod_{t\in T}\Aa_t$.
	\item[(ii)] The image $\ev_t: \Xx \to \Aa_t$ of each evaluation map is norm-dense.
    \item[(iii)] Pointwise multiplication by $C_0(T)$ leaves $\Xx$ invariant.
    \item[(iv)] For each $x\in \Xx$ one has $\sup_{t\in T}\lVert \ev_t(x)\rVert<\infty$.
\end{enumerate}
Then the completion of $\Xx$ w.r.t. the $C^*$-norm
\begin{equation}
\label{eq:sup_norm} 
\lVert x\rVert = \sup_{t\in T} \lVert \ev_t(x)\rVert
\end{equation}
defines a bundle of $C^*$-algebras $\Aa$ with fibers $\Aa_t$. If in addition the norm function \begin{equation}
	    \label{eq:norm_function}
        t\in T\mapsto \lVert \ev_t(x)\rVert
	\end{equation}
	is in $C_0(T)$ for each $x\in \Xx$ then $\Aa$ is a continuous field of $C^*$-algebras.
\end{proposition}
\begin{proof}
The completion is well-defined since the $C^*$-norm is finite by {\it (iv)}. The homomorphisms $\ev_t$ and the $C_0(T)$-action extend continuously to $\Aa$ by definition of the norm \eqref{eq:sup_norm}. Due to {\it (ii)} the extension of $\ev_t$ is surjective, since the images of $C^*$-algebras are closed.
\end{proof}
For a field of groupoids $(\Gg, T, p)$ with $\Gg, T$ lcsc Hausdorff spaces we use the $*$-algebra $\Xx=C_c(\Gg,\omega)$ which can be considered to be a subset of $\prod_{t\in T} C_c(\Gg_t, \omega\rvert_{\Gg_t})$. The restriction maps $\ev_t: C_c(\Gg,\omega)\to C_r^*(\Gg_t,\omega\rvert_{\Gg_t})$ have dense range since every element of $C_c(\Gg_t)$ has a pre-image in $C_c(\Gg)$ (the conditions of the Tietze extension theorem are satisfied as $\Gg$ is lcsc). Since all Haar systems are induced by the counting measure, $\ev_t$ is also a $*$-homomorphism. The completion of this $\Xx$ is $\Aa=C^*_r(\Gg, \omega)$ since \eqref{eq:sup_norm} is exactly the reduced $C^*$-norm. As in \cite{LandsmanContMath2001, Beckus18} the crucial point to verify is therefore merely that the norm function \eqref{eq:norm_function} is continuous.

Lower semi-continuity holds in large generality for reduced groupoid algebras:

\begin{proposition}
\label{prop:lower_semi-cont}
Let $\Gg$ be a lcsc Hausdorff \'etale groupoid with continuous twist $\omega$ which is a continuous field over a locally compact Hausdorff space $T$. For any $a\in C_c(\Gg,\omega)$ the function
$$t \in T\mapsto \lVert \ev_t(a)\rVert_{C^*_r(\Gg_t,\omega\rvert_{\Gg_t})}$$ is lower semi-continuous.
\end{proposition}
\begin{proof}
This is proven in \cite[Theorem 5.5]{LandsmanContMath2001} for a trivial twist and the argument applies just as well for a continuous cocycle $\omega$ (see also \cite{Beckus18} for more details).
\end{proof}

The main problem is therefore proving upper semi-continuity. For this, we follow a strategy vaguely inspired by an argument in \cite{Bellissard94} (which directly proves spectral continuity for certain twisted crossed product algebras) and reduce the upper semi-continuity to an easier to prove special case:
\begin{definition}
\label{def:cont_at_zero}
Let $\Xx$ be a set of sections of $(\Aa_t)_{t\in T}$ satisfying (i-iv)  of Proposition~\ref{prop:fields_by_sections} and let $\Aa$ be the $C^*$-algebra generated by $\Xx$.

We say that $a\in \Aa$ is upper semi-continuous at zero if for every $t_0 \in T$ such that $$\ev_{t_0}(a)=0,$$
one has $\lim_{t\to t_0} \lVert \ev_t(a)\rVert = 0$.
\end{definition}
We will then want to extrapolate from that property to full upper semi-continuity using functional calculus. The difficulty is that functional calculus  of self-adjoint elements of $\Xx$ will generally leave that subalgebra. 
\begin{proposition}
\label{prop:cont_approx_smooth}
Let $\Xx$ be a set of sections of $(\Aa_t)_{t\in T}$ satisfying (i-iv)  of Proposition~\ref{prop:fields_by_sections} and let $\Aa$ be the $C^*$-algebra generated by $\Xx$. Assume that
	\begin{enumerate}
	\item[(i)] all elements of $\Xx$ are upper semi-continuous at zero.
    \item[(ii)] $\Xx$ is closed under the smooth functional calculus of $\Aa$, i.e. $h=h^* \in \Xx$ implies $\varphi(h)\in \Xx$ for all smooth functions $\varphi\in C_c^\infty(\RM)$.
	\end{enumerate}
	For all self-adjoint $h\in \Xx$ the function $t\in T \mapsto \lVert \ev_t(h)\rVert$ is then upper semi-continuous.
\end{proposition}
\begin{proof}
Assume for a contradiction that upper semi-continuity does not hold for $h\in \Xx$ at some $t\in T$. That means there exists a sequence $(t_n)_{n\in \NM}$ converging to $t$ such that 
$$C := \lim_{n\to \infty} \lVert \ev_{t_n}(h)\rVert  > \lVert \ev_{t}(h)\rVert + \delta$$
for some $C>\delta>0$. Since the operator norm is given by the spectral radius,  the spectrum of each $\ev_{t_n}(h)$ eventually has non-trivial intersection with the open set $U=B_{\delta/3}(C)\cup B_{\delta/3}(-C)\subset \RM$. Choose a smooth function $\varphi\in C^\infty_c(\RM,\RM)$ equal to $1$ on $U$, but vanishing in $[-C+\frac{2\delta}{3},C-\frac{2\delta}{3}]$. By the spectral mapping theorem one has $\norm{\varphi(\ev_{t_n}(h))}\geq 1$, but $\varphi(\ev_{t}(h))=0$. 

On the other hand, $\varphi(\ev_{t_n}(h))=\ev_{t_n}(\varphi(h))$ since continuous functional calculus commutes with homomorphisms of $C^*$-algebras. Due to $\varphi(h)\in \Xx$ property (i) implies $\lim_{n\to\infty} \lVert \varphi(\ev_{t_n}(h)) \rVert = 0$, which is a contradiction.
\end{proof}

\begin{remark}
For purposes of K-theory one commonly uses dense spectral invariant subalgebras which are closed under holomorphic functional calculus \cite{Bla}. That appears to be not sufficient in the present case, however, instead of smooth functions it would also be enough if $\Xx$ was closed under a different class of functions which has compactly supported bump functions, for example the Gevrey classes. 
\end{remark}
For \'etale groupoids upper semi-continuity at zero holds for all compactly supported functions:
\begin{proposition}
\label{prop:cont_at_zero}
Let $\Gg$ be a lcsc Hausdorff \'etale groupoid which is a continuous field over a locally compact Hausdorff space $T$, then any $f\in \Xx= C_c(\Gg)$ is upper semi-continuous at zero.
\end{proposition}
\begin{proof}
Assume $\ev_{t_0}(f)=0$ for some $t_0\in T$. Since $f$ is compactly supported and $\Gg$ \'etale there is a cover of $\mathrm{supp}(f)$ by a finite number of open bisections $U_1,...,U_N \subset \Gg$. In the sum on the right-hand side of $$\lVert \ev_t(f)\rVert_{I,s} = \sup_{u\in \Gg_t^{(0)}} \sum_{\gamma\in (\Gg_t)_u} \lvert f(\gamma)\rvert$$ there are therefore at most $N$ non-vanishing terms for any given $u$. Due to compactness of the support and $\ev_{t_0}(f)=0$ there is for each $\epsilon>0$ an open neighborhood $V$ of $\Gg_{t_0} \cap \mathrm{supp}(f)$ on which $\lvert f(\gamma) \rvert<\epsilon N^{-1}$ for all $\gamma \in V$. Note that $p(\mathrm{supp}(f) \setminus V)$ is a compact subset of $T$ which does not contain $t_0$, hence  there exists a neighborhood $S\subset T$ of $t_0$ such that $V$ contains all sets $\Gg_{t} \cap \mathrm{supp}(f)$ for $t\in S$. We conclude
$$\lVert \ev_t(f)\rVert_{I,s}\leq \sup_{u\in \Gg_t^{(0)}} \sum_{\gamma\in(\Gg_t)_u} \epsilon N^{-1} < \epsilon$$
for all $t\in S$. This shows convergence of the semi-norm $\lVert \cdot \rVert_{I,s}$; applying the same argument to $f^*$ shows convergence in the semi-norm $\lVert \cdot \rVert_{I,r}$ and thus also the operator norm.
\end{proof}
We can relax the assumption on compact support:
\begin{definition}
\label{def:zeropres}
We say that $f\in C^*_r(\Gg,\omega)$ has a zero-preserving approximation if there exists a sequence $(f_n)_{n\in \NM}$ in $C_c(\Gg)$ such that
\begin{enumerate}
	\item[(i)] $\lim_{n\to\infty} f_n = f$ in operator norm.
	\item[(ii)] if $\ev_t(f)=0$ for any $t\in T$ then $\ev_t(f_n)=0$ for all $n\in \NM$.
\end{enumerate}
\end{definition}
The purpose of this definition lies in the extension of Proposition~\ref{prop:cont_at_zero}:
\begin{proposition}
\label{prop:cont_at_zero_approx}
Let $\Gg$ be as in Proposition~\ref{prop:cont_at_zero}. If $f\in C^*_r(\Gg,\omega)$ has a zero-preserving approximation then it is upper semi-continuous at zero.
\end{proposition}
\begin{proof}
If $\ev_{t_0}(f)=0$ for some $t_0\in T$ then we can for any $\epsilon>0$ choose first $n$ large enough and then a neighborhood of $t_0$ small enough such that
$$\lVert \ev_t(f)\rVert \leq \lVert \ev_t(f_n)\rVert + \frac{\epsilon}{2} \leq \frac{\epsilon}{2} + \frac{\epsilon}{2}$$
by applying Proposition~\ref{prop:cont_at_zero} to an approximating sequence $(f_n)_{n\in \NM}$.
\end{proof}

Proving upper semi-continuity of the norm therefore reduces to finding a dense subalgebra of $C^*_r(\Gg,\omega)$ that is closed under smooth functional calculus and each element of which admits a zero-preserving approximation.

\begin{remark}
Dropping the structure of a continuous field one can more generally discuss which $f\in C^*_r(\Gg,\omega)$ have an approximation by $f_n \in C_c(\Gg,\omega)$ with $\mathrm{supp}(f_n)\subset \mathrm{supp}(f)$. In some cases every $f\in C^*_r(\Gg,\omega)$ admits an approximation of that form, namely when $\Gg$ is amenable \cite{BrownEtAl2021,BrownEtAl2024} or if $\Gg$ satisfies the rapid decay property with respect to a negative-definite length function \cite{FullerKarmakar}. In that case Proposition~\ref{prop:fields_by_sections} applies immediately with $\Aa=\Xx$, which recovers in particular the special case Theorem~\ref{th:amenable}.
\end{remark}

\section{The rapid decay property}
\label{Sec:RapidDecay}

The rapid decay condition was first considered by Haagerup \cite{Haagerup79} in the context of free groups and formalized by Jolissaint \cite{Jolissaint} for more general reduced group $C^*$-algebras. Relevant for us is its generalization to (twisted) groupoids introduced in \cite{Hou,Weygandt}.

\begin{definition}
A length function $l: \Gg \to \RM_+$ is a continuous function such that $l(\gamma \gamma') \leq l(\gamma)+l(\gamma')$ for all $(\gamma,\gamma')\in \Gg^{(2)}$, $l(\gamma)=l(\gamma^{-1})$ for all $\gamma \in \Gg$ and $l(u)=0$ for all units $u\in \Gg^{(0)}$.
\end{definition}

Fixing a length function one can define weighted $L^2$-norms on $C_c(\Gg)$
\begin{align*}
\lVert f\rVert_{2,s,q} &:= \sup_{u\in \Gg^{(0)}} \left( \sum_{\gamma \in \Gg_u} \lvert f(\gamma)\rvert^2 (1+l(\gamma))^{2q}\right)^{1/2},\\
\lVert f\rVert_{2,r,q} &:= \sup_{u\in \Gg^{(0)}} \left( \sum_{\gamma \in \Gg^u} \lvert f(\gamma)\rvert^2 (1+l(\gamma))^{2q}\right)^{1/2},\\
\lVert f\rVert_{2,q} &:= \max(\lVert f\rVert_{2,s,q},\lVert f\rVert_{2,r,q}).
\end{align*}
The completion of $C_c(\Gg)$ w.r.t. to $\lVert \cdot \rVert_{2,p}$ is the Banach space $H^p(\Gg)$, the completion w.r.t. to the family of norms $(\lVert \cdot \rVert_{2,p})_{p\in \NM}$ the Fr\'echet space $H^\infty(\Gg)$.

\begin{definition}
One says that $\Gg$ has the $\omega$-twisted rapid decay property if there exist constants $C>0$, $p\geq 0$ such that
$$\lVert f \rVert_{C_r^*(\Gg, \omega)}\leq C \lVert f\rVert_{2,p}$$
for all $f\in C_c(\Gg)$ under the natural inclusion $C_c(\Gg) \subset C_r^*(\Gg, \omega)$.
\end{definition}
If $\Gg$ has the $\omega$-twisted rapid decay property then $H^\infty(\Gg)$ can be identified with a dense subset of $C_r^*(\Gg, \omega)$, as one always has $\norm{f}_{2,s,0}\leq \norm{f}_{C^*_r(\Gg,\omega)}$ for \'etale groupoids \cite[Proposition II.4.1]{RenaultBook}, showing that the extension of the inclusion to $H^\infty(\Gg)$ is injective. The convolution multiplication of $C_c(\Gg,\omega)$ is jointly continuous w.r.t. the semi-norms, indeed, adapting \cite[Lemma 1.2.4]{Jolissaint} to the case of groupoid algebras one has
$$\lVert f_1 * f_2\rVert_{2,q}\leq C \lVert f_1\rVert_{2,p+q} \, \lVert f_2\rVert_{2,q}$$
with the constants $C$, $p$ from the rapid decay property. Hence one obtains a Fr\'echet algebra that we denote $H^\infty(\Gg,\omega)$. 

\begin{proposition}[{\cite{Weygandt,Chatterji}}]
\label{prop:twisted_rapid_decay}
If the groupoid $\Gg$ satisfies the rapid decay property for trivial twist then it also satisfies the $\omega$-twisted rapid decay property for any continuous twist and with the same constants $C$ and $s$.
\end{proposition}
\begin{proof}
Applying the triangle inequality to the induced representation one has
$$\lVert \pi_u^\omega(f) g\rVert_{\ell^2(\Gg_u)}\leq \lVert \pi_u^{1}(\lvert f\rvert)\, \lvert g\rvert \rVert_{\ell^2(\Gg_u)} \leq \lVert \,\lvert f\rvert\, \rVert_{C^*_r(\Gg)} \lVert g\rVert_{\ell^2(\Gg_u)}$$
and the analogous inequality for $\Gg^u$ where we distinguish between the $\omega$-twisted induced representation and the trivially twisted one. Here $\lvert f\rvert$ is the point-wise modulus. Applying the untwisted rapid decay property to the right-hand side bounds the operator norm of $\pi^\omega_u(f)$ and completes the proof.
\end{proof}

The main result which we will need is this:
\begin{theorem}
Let $\Gg$ be a lcsc Hausdorff \'etale groupoid with the $\omega$-twisted decay property and identify the Fr\'echet algebra $H^\infty(\Gg, \omega)$ with a dense subalgebra of $C_r^*(\Gg, \omega)$.
Then
\begin{enumerate}
	\item[(i)] the inclusion $H^\infty(\Gg, \omega)\to C_r^*(\Gg, \omega)$ is spectral invariant, i.e. $$\sigma_{H^\infty(\Gg, \omega)}(f)=\sigma_{C_r^*(\Gg,\omega)}(f)$$
	for all $f\in H^\infty(\Gg, \omega)$, equivalently if $f\in H^\infty(\Gg,\omega)$ is invertible in $C^*_r(\Gg,\omega)$ then $f^{-1}\in H^\infty(\Gg,\omega)$.
	\item[(ii)] $H^\infty(\Gg, \omega)$ is closed under the functional calculus of its self-adjoint elements, i.e. for each $f=f^* \in H^\infty(\Gg, \omega)$ and $\varphi\in C_c^\infty(\RM)$ the element $\varphi(f)\in C_r^*(\Gg, \omega)$ is in $H^\infty(\Gg, \omega)$.
\end{enumerate}
\end{theorem}
\begin{proof}
Point {\it (i)} is \cite[Theorem 4.2]{Hou}, and its proof generalizes immediately to the twisted case, as has been pointed out by \cite{Weygandt}. While only the closure under holomorphic functional calculus has been remarked upon (due to its importance to $K$-theory) the method of proof allows one to conclude more strongly closure under smooth functional calculus. Let us recall therefore some details from the proof.

Define on each $\ell^2(\Gg_u)$ a densely defined unbounded self-adjoint multiplication operator acting by $W_u \delta_\gamma \mapsto l(\gamma) \delta_\gamma$ on the standard basis $(\delta_\gamma)_{\gamma\in \Gg_u}$. For all bounded operators $T\in \Bb(\ell^2(\Gg_u))$ which preserve the domain of $W_u$ and such that the commutator $[W_u,T]$ extends to a bounded operator define $\delta_u(T) = \imath [W_u,T]\subset \Bb(\ell^2(\Gg_u))$. This results in a norm-closed derivation $\delta_u: \mathrm{Dom}(\delta_u)\subset \Bb(\ell^2(\Gg_u))\to \Bb(\ell^2(\Gg_u))$. Setting
\begin{align*}
\lVert T\rVert_{u,n} &= \lVert \delta_u^n(T)\rVert, \qquad &T\in \Bb(\ell^2(\Gg_u))\\
\tnorm{f}_n &= \sup_{u\in \Gg^{(0)}} \norm{\pi^\omega_u(f)}_{u,n}, \qquad &f\in C^*_r(\Gg,\omega).
\end{align*}
one can show using the rapid decay property (cf. the proof of \cite[Theorem 4.2]{Hou}) that $(\tnorm{\cdot}_n)_{n\in \NM}$ defines on $C_c(\Gg,\omega)$ a family of semi-norms equivalent to $(\lVert \cdot\rVert)_{2,n\in \NM}$.

Since the family $(\tnorm{\cdot})_{n\in \NM}$ are so-called differential semi-norms one can conclude that $H^\infty(\Gg,\omega)$ is spectral invariant in its $C^*$-completion and closed under smooth functional calculus \cite{BlackadarCuntz92}. 

Let us give a direct proof of the latter. It is well-known that if $T,S\in \mathrm{Dom}(\delta)$ lie in the domain of a closed derivation then $TS, T^{-1}\in \mathrm{Dom}(\delta)$ with
$$\delta(TS)=- T \delta(S) + \delta(T) S$$
and
\begin{equation}
\label{eq:leibniz}
\delta(T^{-1})=- T^{-1} \delta(T) T^{-1}
\end{equation}
and by iteration one obtains similar expressions for higher derivatives. If an element $f\in H^\infty(\Gg,\omega)$ has an inverse $a^{-1} \in C_r^*(\Gg,\omega)$ then all $\pi^\omega_u(f)$ are invertible and $$\lVert f^{-1}\rVert = \sup_{u \in \Gg^{(0)}}\lVert \pi^\omega_u(f)^{-1}\rVert < \infty$$ by definition of the norm on $C^*_r(\Gg,\omega)$. Each $\pi^\omega_u(f)^{-1}$ is in the domain of $\delta_u$ as pointed out above and due to \eqref{eq:leibniz} one has a uniform norm bound
$$\lVert \delta_u(\pi^\omega_u(f)^{-1})\rVert \leq \lVert \pi^\omega_u(f)^{-1}\rVert^2 \, \lVert \delta_u(\pi^\omega_u(f))\rVert \leq \lVert f^{-1}\rVert_{C_r^*(\Gg,\omega)}^2 \, \tnorm{f}_{1}.$$
In particular, $\tnorm{f^{-1}}_1<\infty$ is finite and iterating the argument for higher derivatives shows $f^{-1} \in H^\infty(\Gg,\omega)$.

For every $\varphi\in C_c^\infty(\RM)$ and self-adjoint operator $H$ one can write $\varphi(H)$ as a norm-convergent integral using the Dynkin-Helffer-Sjostrand formulation of the smooth functional calculus:
\begin{equation}
\label{eq:hs_formula}
\varphi(H) = \frac{1}{2\pi}\int_{D} \partial_{\overline{z}} \tilde{\varphi}(z)\frac{1}{H+z} \mathrm{d}z\wedge \mathrm{d}\overline{z}
\end{equation}
where $D\subset \CM$ is a compact region and $\tilde{\varphi}:D \to \CM$ is a smooth almost analytic function such that $\varphi(\lambda)=\tilde{\varphi}(\lambda)$ for all $\lambda\in \RM$ and such that for all $K>0$ there exists a constant $C_K$ such that $\lvert\partial_{\overline{z}} \tilde{\varphi}(x+\imath y)\rvert \leq C_K \lvert y\rvert^{K}$ (see \cite[Appendix C.2]{DerezinskiGerard} for this precise form).

With the trivial resolvent estimate $\lVert (\pi^\omega_u(h)+z)^{-1}\rVert \leq \lvert \Im m \, z\rvert^{-1}$ and iterating \eqref{eq:leibniz} one can show that there is for each $n$ a constant $c_n$, independent of $u$, such that
\begin{equation}
\label{eq:resolvent_bound}
\lVert \delta_u^{n}\left(\frac{1}{\pi^\omega_u(h)+z}\right) \rVert \leq c_n \max_{\sum_i {m_i=n}}\left(\prod_{i}\lVert \delta_u^{m_i}(h)\rVert\right) \max\left(1,\lvert \Im m\, z\rvert^{-(n+1)}\right)
\end{equation}
where the maximum is taken over all ways to decompose $n$ as a sum with strictly positive terms $m_i$.

Since each $\delta^n_u$ is a closed operator, we can prove that $\varphi(h)\in H^\infty(\Gg,\omega)$ by showing that the integral on the right-hand side of \eqref{eq:hs_formula} converges for each $n$ in the semi-norm $\sup_{u\in \Gg^{(0)}}\norm{\cdot}_{u,n}$. That is an immediate consequence of \eqref{eq:resolvent_bound} and dominated convergence applied to the right-hand side of
$$\delta_u^{n}(\pi^\omega_u(\varphi(h))) = \frac{1}{2\pi}\int_{D} \partial_{\overline{z}} \tilde{\varphi}(z) \delta_u^{n}\left(\frac{1}{\pi^\omega_u(h)+z}\right) \mathrm{d}z\wedge \mathrm{d}\overline{z}$$
which is then for each $n$ a norm-convergent integral with norm-bound independent of $u$ and, due to closedness of $\delta^n_u$, equal to the left-hand side. 
\end{proof}

Our remaining task is to prove that every element of $H^\infty(\Gg, \omega)$ has a zero-preserving approximation. That could be deduced from the recent paper \cite{FullerKarmakar} but let us still give a proof for self-containedness and since we have a short alternative argument:
\begin{proposition}
\label{prop:approximation_rapid_decay}
Let $(\Gg, T, p)$ be a field of groupoids with $\Gg$ a lcsc Hausdorff \'etale groupoid that satisfies the rapid decay property w.r.t. a length function $l$. Then any $f\in H^\infty(\Gg,\omega)$ has a zero-preserving approximation w.r.t. the norm of $C^*_r(\Gg,\omega)$.
\end{proposition}
\begin{proof}
Due to the rapid decay property we just need to find for any given $\epsilon>0$ some function $\tilde{f}\in C_c(\Gg)$ with $\lVert{f-\tilde{f}}\rVert_{2,p}<\epsilon$ and $\mathrm{supp}(\tilde{f})\subseteq \mathrm{supp}(f)$. By density of $C_c(\Gg)$ there is some $g\in C_c(\Gg)$ with $\norm{f-g}_{2,p}<\frac{1}{2}\epsilon$. Let $\chi\in C_c(\Gg)$ be a function with $0\leq \chi \leq 1$ which is equal to $1$ on $\mathrm{supp}(g)$.
Any $f \in H^\infty(\Gg)$ is in particular a continuous function and with $\tilde{f}(\gamma) := f(\gamma) \chi(\gamma)$ one has
$$\lVert f - \tilde{f}\rVert_{2,p} \leq \lVert{f - g}\rVert_{2,p} + \lVert{\tilde{f} - g}\rVert_{2,p}.$$
Bounding the second term using
\begin{align*}
    \lVert\tilde{f} - g\rVert_{2,s,p}^2 &= \sup_{u\in \Gg^{(0)}} \sum_{\gamma \in \Gg_u} \lvert f(\gamma) \chi(\gamma) - g(\gamma)\rvert^2 (1+l(\gamma))^{2p}\\
    &= \sup_{u\in \Gg^{(0)}} \sum_{\gamma \in \Gg_u} \lvert (f(\gamma)  - g(\gamma))\chi(\gamma)\rvert^2 (1+l(\gamma))^{2p} \leq \norm{f-g}^2_{2,s,p}
\end{align*}
and its equivalent for the $\norm{\cdot}_{2,r,p}$-norm we obtain $\lVert{f-\tilde{f}}\rVert_{2,p}\leq 2\norm{f-g}_{2,p}<\epsilon$.
\end{proof}

We can now conclude our main result:
\begin{proof} (of Theorem~\ref{th:main})
By Proposition~\ref{prop:approximation_rapid_decay} and Proposition~\ref{prop:cont_at_zero_approx} every element of $\Xx=H^\infty(\Gg,\omega)$ is upper semi-continuous at zero and $H^\infty(\Gg,\omega)$, hence the conditions of Proposition~\ref{prop:cont_approx_smooth} are satisfied. Together with Proposition~\ref{prop:lower_semi-cont} the norm function $t\in T\mapsto \norm{\ev_t(a)}$ is therefore continuous for every $a\in \Xx$. Using Proposition~\ref{prop:fields_by_sections} we conclude that the completion $\Aa=C_r^*(\Gg,\omega)$ of $\Xx$ is a continuous field of $C^*$-algebras.
\end{proof}

For fields of groupoids we will usually want to use the following obvious sufficient condition:
\begin{proposition}
\label{prop:uniform_rapid_decay}
Let $(\Gg, T, p)$ be a field of groupoids with $\Gg$ a lcsc Hausdorff \'etale groupoid and $l$ a length function on $\Gg$. For each $t\in T$ there is a length function $l_t = l \circ p\rvert_{\Gg_t}$ on $\Gg_t$. If each of the fibers $\Gg_t$ has the rapid decay property 
$$\lVert \ev_t(a) \rVert_{C_r^*(\Gg_t, \omega\rvert_{\Gg_t})}\leq C \lVert \ev_t(a)\rVert_{2,p}$$
with $C,p>0$ independent of $a\in C_c(\Gg)$ and $t\in T$ then $\Gg$ has the rapid decay property.
\end{proposition}
\begin{proof}
As sets we can think of $\Gg$ as a disjoint union $\Gg = \sqcup_{t\in T} \Gg_t$ of invariant subgroupoids and likewise for the unit spaces $\Gg^{(0)} = \sqcup_{t\in T} \Gg^{(0)}_t$. Expanding the definitions one has
$$\lVert a \rVert_{2,s,p} = \sup_{t\in T} \sup_{u\in \Gg_t^{(0)}} \left( \sum_{\gamma \in (\Gg_t)_u} \lvert \ev_t(a)(\gamma)\rvert^2 (1+l_t(\gamma))^{2p}\right)^{1/2} = \sup_{t\in T} \left \lVert \ev_t(a)\right\rVert_{2,s,p}$$
and similarly for the seminorm $\norm{\cdot}_{2,r,p}$ as well as
$$\lVert a\rVert_{C^*_r(\Gg,\omega)} = \sup_{t\in T} \left\lVert  \ev_t(a)\right \rVert_{C^*_r(\Gg_t,\omega\rvert_{\Gg_t})}$$
which clearly implies
$\lVert a \rVert_{C_r^*(\Gg, \omega)}\leq C \lVert a\rVert_{2,p}$.
\end{proof}

\section{Examples}
\label{sec:examples}

\subsection{Fields of finite approximating groups}

Let $G$ be a discrete group and $(H_n)_{n\in \NM}$ a decreasing sequence of normal finite-index subgroups $H_n \supset H_{n+1}$ such that $\bigcap_{n\in \NM} H_n = \{e\}$. Denote $G_n=G/H_n$ and $G_\infty=G$ to define the field of groupoids 
$$\Gg = \bigsqcup_{n\in \overline{\NM}} \{n\} \times G_n$$
over the one-point compactification $\overline{\NM}=\NM \cup \{\infty\}$ with fiber-wise operations. Groupoids of this form are sometimes called Higson-Lafforgue-Skandalis (HLS) groupoids after the authors of \cite{HigsonLafforgueSkandalis2002} where they were introduced. One can equivalently think of $\Gg$ as $G \times \overline{\NM}/\sim$ with the relation $$(n, g)\sim (m,\tilde{g}) \, \iff n=m \; \text{ and } gH_n = \tilde{g}H_m$$
which endows $\Gg$ with a quotient topology that makes it into a locally compact Hausdorff groupoid \cite{HigsonLafforgueSkandalis2002, Willett}. The map $p: (n,g) \mapsto n$ is a continuous open surjection, hence one has a field of groupoids.

In general $G$ is non-amenable, i.e. the quotient map $\pi_n: G\to G_n$ does not necessarily give rise to a homomorphism $C_r^*(G) \to C_r^*(G_n)$. However, for finite linear combinations $a=\sum_{g\in G} a_g g \in \CM G$ in the group algebra  one can canonically define an approximating sequence $\pi_n(a)=\sum_{g\in G} a_g \pi_n(g)\in \CM G_n$. Under the identification $\CM G_n \simeq C_c(G_n)$ we can think of $a$ in terms of a compactly supported coefficient function $a:G \to \CM$. For any $a\in \CM G$ the elements $(\pi_n(a))_{n\in \NM}$ combine in that sense to a continuous function $f_a\in C_c(\Gg)$ collecting all finite-dimensional approximations.

From Proposition~\ref{prop:uniform_rapid_decay} we have:

\begin{proposition}
If there is a length function on $\Gg$ such that the rapid decay property holds uniformly in $n\in \overline{\NM}$ for the quotient groups $G/H_n$ in the sense of Proposition~\ref{prop:uniform_rapid_decay} then $\Gg$ has the rapid decay property and 
\begin{enumerate}
	\item[(i)] $C_r^*(\Gg)$ is a continuous field of $C^*$-algebras over $\overline{\NM}$.
	\item[(ii)] If there is for some $a=a^*\in C_r^*(G)$ a section $\hat{a}\in C_r^*(\Gg)$ with $\ev_\infty(\hat{a}) = a$ then
	$$\lim_{n\to \infty}\sigma(\ev_n(\hat{a})) \to \sigma(a)$$
	in the Hausdorff topology.
\end{enumerate}
\end{proposition}

If $G$ is finitely generated one can put the word-length metric on each $G_n$ to obtain a continuous length function (continuity at $\infty$ follows from $\bigcap_{n\in \NM} H_n = \{e\}$). Having an approximation by finite groups with uniform rapid decay would clearly be very useful for spectral computations involving group algebras as they are described in e.g. \cite{Lux2023}. However, while uniform rapid decay does hold for some groups (e.g. the free abelian groups $G=\ZM^d$, $H_n=n\ZM^d$) we are not aware of a non-amenable group which satisfies it. Indeed, spectral continuity allows us to explicitly rule out existence in certain cases:

\begin{example}
For a counter-example let $F_2$ be the free group of two generators $x,y$ and choose any decreasing sequence of finite-index subgroups. Free groups have the rapid decay property \cite{Haagerup79} (cf. \cite[Example 1.2.3]{Jolissaint}). The element  $a = x+x^{-1}+y+y^{-1}$ acting on $\ell^2(F_2)$ in the regular representation is equal to the adjacency operator of the Cayley graph. The spectrum of the latter is given by the interval $[-2\sqrt{3},2\sqrt{3}]$ \cite{Kesten,McKayLG1981}. For any sequence of normal subgroups one defines the finite-dimensional approximations $\pi_n(a)= \pi_n(x)+\pi_n(x^{-1})+\pi_n(y)+\pi_n(y^{-1})$ acting on the finite-dimensional Hilbert spaces $\ell^2(F_2/H_n)$. However, $4\in \sigma(\pi_n(a))$ as the constant function $1_{F_2/H_n}$ is an eigenvector and hence the spectra $\sigma(\pi_n(a))$ do not converge to $\sigma(a)$.  In conclusion, no approximation by finite quotient groups can satisfy the rapid decay property uniformly.
\end{example}

\begin{remark}
This example shows that it is not enough for the endpoint to have the rapid decay property. In contrast, if $G$ is amenable, then $C^*_r(\Gg)$ always decomposes as a continuous field independent of the approximating sequence of subgroups. This is a consequence of \cite[Corollary 5.6]{LandsmanContMath2001} which shows that if $\Gg$ is a field of groupoids over a space $T$ then the norms of any section are continuous at any fiber $t$ where $\Gg_t$ is amenable.
\end{remark}

\begin{remark}
It is tempting to think that in Theorem~\ref{th:amenable} one can drop amenability to the weaker condition $C^*_r(\Gg) = C^*(\Gg)$, however, this is not true: In the example above with the free group $F_2$ it is known that $C^*_r(\Gg)=C^*(\Gg)$ \cite{Willett} thus forming a counterexample. To exhibit the underlying problem note that for the reduced norm $n\in \overline{\NM}\mapsto \norm{\ev_n(a)}_{C_r^*(G_n)}$ is lower semi-continuous and for the full norm $n\in \overline{\NM}\mapsto \norm{\ev_n(a)}_{C^*(G_n)}$ is upper semi-continuous for all $a\in C_c(\Gg)$ \cite{LandsmanContMath2001}. However, the norms of $C^*_r(F_2)$ and $C^*(F_2)$ are not equivalent and thus one cannot conclude continuity of either norm function at $\infty$. For a true generalization of Theorem~\ref{th:amenable} one requires more strongly that full and reduced norms coincide on $C_c(\Gg_t)$ for each fiber.
\end{remark}
\begin{remark}
That this construction does not yield a continuous field for $G=F_2$ can be seen already from \cite{HigsonLafforgueSkandalis2002}, where it is shown that the diagram
$$0 \to C_r^*(\Gg\rvert_{\NM})\to C_r^*(\Gg)\to C_r^*(F_2)\to 0$$
induced by the restriction of unit spaces fails to be exact, but it would have to be if $C_r^*(\Gg)$ was a continuous field over $\overline{\NM}$ with fibers $C^*_r(G_n)$.
\end{remark}

\subsection{Twisted group algebras and hyperbolic crystals}

Let $G$ be a discrete group and $T$ a compact Hausdorff space. A continuous family of twists is a family of $2$-cocycles $(\omega_t)_{t\in T}$ where the functions $\omega_t: G\times G\to \Uu(\CM)$ extend to a continuous function $\omega: \Gg^{(2)}\to \Uu(\CM)$ with $\Gg=T \times G$ and the composable pairs $\Gg^{(2)}=\{((t_1,g_1),(t_2,g_2)): \; t_1=t_2\}.$

The \'etale groupoid $\Gg$ is a field of groupoids over $T$. It is easy to show that $C_r^*(\Gg,\omega)\simeq C(T)\rtimes_{r,\omega}G$ is isomorphic to a reduced twisted crossed product with a trivial $G$-action but non-trivial twist. The convolution multiplication of $C_c(\Gg)$ can therefore be represented conveniently in terms of formal Fourier series: The functions $u_g(t,\gamma)=\delta_{g,\gamma}$ define unitary elements $u_g\in C_r^*(\Gg,\omega)$ for all $g\in G$ and we can write any $f\in C_c(\Gg)$ uniquely as a series 
$$f = \sum_{g\in G} f_g \, u_g$$
with finitely many non-vanishing coefficients $f_g\in C(T)$. The unitaries $U_g$ commute with $C(T)$ and their multiplication law is
$$u_g u_h = \omega(g,h) u_{gh}$$
where $\omega(g,h)$ is interpreted as a function in $C(T)$. 

The following is obvious from Proposition~\ref{prop:twisted_rapid_decay} and Proposition~\ref{prop:uniform_rapid_decay}:
\begin{proposition}
\label{prop:twists_field}
If $G$ has the rapid decay property then $\Gg=T\times G$ has it as well and hence $C^*_r(\Gg,\omega)$ decomposes as a continuous field of twisted group algebras.
\end{proposition}

As an application we consider hyperbolic crystals in a magnetic field. A hyperbolic crystal $\Gamma$ shall be a discrete subset of the hyperbolic plane $\mathbb{H}^2$ corresponding to the orbit of a point under some Fuchsian group $G$:
\begin{definition}\label{def:hyp-crystal}
	Let \(G \le \mathrm{Isom}^{+}(\mathbb{H}^{2}) \cong \mathrm{PSL}(2,\mathbb{R})\) be a
	cocompact Fuchsian group, i.e.\ a discrete subgroup whose
	quotient \(\mathbb{H}^{2}/G\) is a compact surface.
	For a fixed base point \(x \in \mathbb{H}^{2}\) we call the orbit
	\[
		\Gamma \;=\; G \cdot x \;=\;
		\{ g \cdot x \mid g \in G \}
	\]
	a hyperbolic crystal.
	
	\smallskip
	\noindent
	Equivalently, a hyperbolic crystal is a discrete subset
	\(\Gamma \subset \mathbb{H}^{2}\) on which some cocompact Fuchsian group acts freely and transitively.
\end{definition}
Any such group has the rapid decay property \cite{Jolissaint}.
For applications in physics one studies tight-binding models for self-adjoint Hamiltonians which give rise to dynamics on the Hilbert space $\ell^2(\Gamma)$ (e.g. \cite{KollarNature2019} for experimentally relevant physical models, but similar ones have also been studied previously in \cite{CareyCMP1998, Carey2006, MarcolliMathai2006} due to a conjectured relation between fractional quantum Hall phases and hyperbolic geometry). 

A simple model is the discrete Laplacian on $\Gamma = G \cdot x$, where $x$ has trivial stabilizer under $G$ and hence $\Gamma$ may be identified with the vertices of any Cayley graph of $G$. The graph structure is then fixed by choosing a finite symmetric generating set $S\subset G$. The magnetic graph Laplacian takes the form
$$H^{\omega} = \sum_{\gamma \in S} U^{\omega}_\gamma$$
where $U^\omega: G \to \Bb(\ell^2(\Gamma))$ is a projective unitary representation of $G$ with a twisting $2$-cocycle such that
$$U_\gamma U_{\gamma'}=\omega(\gamma, \gamma')U_{\gamma \gamma'}.$$ 
When we identify $\Gamma$ with $G$ then we can think of the (faithful) regular representation $\pi^\omega$ of the group algebra $C^*_r(G,\omega)$ to act on $\ell^2(\Gamma)$. We can then write $H^{\omega}=\pi^\omega(h)$ as the image of the element $h=\sum_{\gamma\in S} \delta_\gamma \in C_c(G)$ which does not depend on $\omega$. As a consequence of Proposition~\ref{prop:twists_field} we therefore find:
\begin{proposition}
The spectrum of $H^\omega$ depends continuously on $\omega\in Z^2(G, \Uu(\CM))$.    
\end{proposition}
Physically, the phase factors \(\omega\) arise from parallel transport with respect to a $U(1)$-connection (magnetic field) on $\mathbb{H}^{2}$. For a constant magnetic field it is given by the {\it area cocycle} $$\omega_\theta(\gamma, \gamma')= \exp(\imath \theta A(\gamma,\gamma'))$$ with $\theta\in \RM$ having the interpretation of the magnetic field strength in some unit system and $A(\gamma,\gamma')$ the area of the hyperbolic triangle with vertices $(x, \gamma\cdot x, \gamma'\cdot x)$. 

From the well-studied example of the Euclidean plane one expects that the spectrum of $H^{\omega_\theta}$ to have a non-trivial dependence on the magnetic field strength $\theta$: there the spectrum consists of a finite collection of intervals if the magnetic field is rational and a Cantor set if it is irrational \cite{AvilaJito2009}. Little appears to be known rigorously for the spectra of discrete magnetic Laplacians of hyperbolic crystals, however, it is reasonable to expect that the twist at least in some cases may open gaps in the spectrum. This is known for the corresponding continuum model of magnetic Laplacians on the hyperbolic plane \cite{Comtet1987} and is also suggested by a recent numerical study which attempts to compute the spectrum using finite-volume approximations \cite{StegmaierPRL2022}.

\end{document}